\documentclass{amsart}

\usepackage[lite,initials,nobysame]{amsrefs}
\usepackage{amssymb}
\usepackage{hyperref}

\numberwithin{equation}{section}

\theoremstyle{plain} 
\newtheorem{theorem}{Theorem}[section]

\newtheorem{lemma}[theorem]{Lemma}

\usepackage{tikz}
\newcommand\vx{\node[circle, draw, fill=black, inner sep=0pt, minimum width=6pt]}

\begin{document}

\title[Abstract Regular Polytopes of Finite Irreducible Coxeter Groups]{Abstract Regular Polytopes of \\ Finite Irreducible Coxeter Groups} 

\author{Malcolm Hoong Wai Chen} \address{Department of Mathematics, University of Manchester, Manchester, United Kingdom.} \email{malcolmhoongwai.chen@postgrad.manchester.ac.uk}
\author{Peter Rowley} \address{Department of Mathematics, University of Manchester, Manchester, United Kingdom.} \email{peter.j.rowley@manchester.ac.uk}


\begin{abstract}

Here, for $W$ the Coxeter group $\mathrm{D}_n$ where $n > 4$, it is proved that the maximal rank of an abstract regular polytope for $W$ is $n - 1$ if $n$ is even and $n$ if $n$ is odd. Further it is shown that $W$ has abstract regular polytopes of rank $r$ for all $r$ such that $3 \leq r \leq n - 1$, if $n$ is even, and $3 \leq r \leq n$, if $n$ is odd. The possible ranks of abstract regular polytopes for the exceptional finite irreducible  Coxeter groups are also determined.

\end{abstract}

\maketitle

\section{Introduction}\label{introduction}

Finite Coxeter groups appear in many guises in the mathematics literature -- as Weyl groups of semisimple Lie algebras, reflection groups, and automorphism groups of regular polytopes to mention a few. It is the last area that is of interest here.  We recall that $\mathrm{Sym}(n+1)$, the symmetric group of degree $n + 1$, also the Coxeter group of type $\mathrm{A}_n$ is the automorphism group of a regular $n$-simplex, while the Coxeter group of type $\mathrm{B}_n$ is the automorphism group of the $n$-cube and its dual, the $n$-cross polytope.

In this paper we examine abstract regular polytopes of the finite irreducible Coxeter groups $W$.  As is documented in McMullen and Schulte \cite{arpbible}, this is equivalent to investigating the C-strings of $W$, and we follow that approach. A \emph{C-string} of a group $G$ is a set of involutions $S = \{s_1, \dots , s_r \}$ of $G$ which generates the group. Additionally, setting $I = \{ 1, \dots, r \}$ and,  for $J \subseteq I, W_J = \langle s_j \;| \;  j \in J \rangle$,  they must satisfy

\begin{enumerate}
\item[(i)] for all $J, K \subseteq I$, $W_J \cap W_K = W_{J \cap K}$  (intersection property); and
\item[(ii)] $s_is_j = s_js_i$ for all $i, j \in I$ with $|i - j| \geq 2$ (string property).
\end{enumerate}

We refer to $r$ as the \emph{rank} of the C-string, equivalently, the rank of the associated abstract regular polytope, and also call $(G,S)$ a \emph{string C-group}. The maximal rank for a C-string of $W$ will be denoted by $r_{\mathrm{max}}(W)$. If $W$ is either of type $\mathrm{A}_n$ or $\mathrm{B}_n$, then their defining presentation in terms of fundamental (or simple) reflections yields a C-string of rank $n$ (see Section 5.5 of \cite{coxbible} for the intersection property).  Moreover $r_{\mathrm{max}}(W) = n$  when $W \cong \mathrm{A}_n$, a consequence of a recent deep result by Whiston \cite{whiston} which uses the classification of finite simple groups.  For $W \cong \mathrm{B}_n$, $r_{\mathrm{max}}(W) = n$ when $n$ is even and $r_{\mathrm{max}}(W) = n + 1$ when $n$ is odd. In the latter case the corresponding C-string is degenerate, meaning the string is disconnected.  See Theorem \ref{transrank}(ii) for more details. If we insist on only considering non-degenerate strings, then we have 
$r_{\mathrm{max}}(W) = n$.  On the other hand, $\mathrm{I}_2(m)$, identified as $\mathrm{Dih}(2m)$ are also well known to only exhibit rank two polytopes given by the $m$-gon. Now in the remaining infinite family of finite irreducible Coxeter groups we have $\mathrm{D}_n$ whose Dynkin diagrams are not strings. Our first result concerns the maximal rank of C-strings for this family.  Noting that $\mathrm{D}_4$ has no C-strings, we have the following theorem.  

\begin{theorem}\label{maximalrank}
Suppose that $W$ is the Coxeter group $\mathrm{D}_n$ with $n \geq 5$. If $n$ is even, then $r_{\mathrm{max}}(W) = n - 1$ and if $n$ is odd, then $r_{\mathrm{max}}(W) = n$.
\end{theorem}

The question as to the existence of abstract regular polytopes for intermediate rank values is the subject of the next theorem.

\begin{theorem}\label{intermediateranks} Suppose that $W$ is the Coxeter group $\mathrm{D}_n$ with $n \geq 5$.  Then $W$ has a C-string of rank $r$ for all $r$ with $3 \leq r \leq r_{\mathrm{max}}(W)$.

\end{theorem}

To complete the picture we consider the exceptional finite irreducible Coxeter groups.

\begin{theorem}\label{exceptional} Suppose that $W$ is an exceptional finite irreducible Coxeter group.
\begin{enumerate}

\item[(i)] If $W$ is one of $\mathrm{I}_2(m), \mathrm{H}_3, \mathrm{H}_4, \mathrm{F}_4$, then $r_{\mathrm{max}}(W)$ is the Coxeter rank of $W$.
\item[(ii)] If $W$ is one of $\mathrm{E}_6, \mathrm{E}_7, \mathrm{E}_8$, then $r_{\mathrm{max}}(W)$ is 5, 6, and 7, respectively.

\end{enumerate}

\end{theorem}

The quest to classify string C-groups started from early experimental results by Hartley \cite{p29} and also the joint work of Leemans and Vauthier \cite{p42} which resulted in atlases of abstract regular polytopes for small groups. More efficient computer algorithms have also been developed which successfully enumerated the C-strings of some sporadic simple groups \cite{p30,p37,p38}. This experimental data has led to some interesting conjectures which were later proven theoretically. In particular, the possible ranks of C-strings has been investigated for some infinite families of almost simple groups, namely the Suzuki groups $\mathrm{Sz}(q)$ \cite{p36}, Ree groups ${}^{2}\mathrm{G}_2(q)$ \cite{p41}, groups with socle $\mathrm{PSL}(2,q)$ \cite{p14,p39,p40},  groups $\mathrm{PSL}(3,q)$ and $\mathrm{PGL}(3,q)$ \cite{p5}, groups $\mathrm{PSL}(4,q)$ \cite{p3}, symmetric groups \cite{p20}, alternating groups \cite{altrank,p22}, orthogonal and symplectic groups \cite{p2,newpaper}. We refer the interested reader to a recent survey article by Leemans \cite{arpsurvey} for more details of these investigations.

Only a few family of groups are known to give rise to string C-groups of arbitrarily large rank.  As already noted, the highest possible rank of a string C-group representation for $\mathrm{Sym}(n)$  is $n-1$, and for $n \ge 12$, the highest rank for $\mathrm{Alt}(n)$ is $\lfloor{\frac{n-1}{2}}\rfloor$ \cite{altrank}. Furthermore, $\mathrm{Sym}(n)$ has C-strings of rank $r$ for every $3 \leq r \leq n-1$ \cite{p20}, and for $n \ge 12$, $\mathrm{Alt}(n)$ also has C-strings of rank $r$ for every $3 \leq r \leq \lfloor{\frac{n-1}{2}}\rfloor$ \cite{p22}. Likewise for all integers $k,m \ge 2$, the orthogonal groups $\mathrm{O}^{\pm}(2m,2^k)$ and the symplectic groups $\mathrm{Sp}(2m,2^k)$ have string C-group representations of rank $2m$ and $2m+1$ \cite{p2}, respectively. Interestingly in \cite{rankred}, Brooksbank proved a `rank reduction theorem' which can be applied on known string C-groups to obtain C-strings of the same group with smaller ranks, and consequently it was shown that $\mathrm{O}^{\pm}(2m,2^k)$ and $\mathrm{Sp}(2m,2^k)$ have string C-group representations of rank $r$ for every $3 \leq r \leq 2m$ and $3 \leq r \leq 2m+1$, respectively. For recent work on polytopes of rank $n/2$ for a transitive group of degree $n$ see \cite{interesting}. Finally, we mention a further study in which two interesting families of C-strings for $\mathrm{B}_n$ are constructed \cite{unravelledbn}.

This paper is structured as follows. In Section \ref{pre}, we  state some preliminary results on Coxeter groups and string C-groups.  Lemma \ref{bndn} plays an important role as our investigations into $W \cong \mathrm{D}_n$ are conducted in $\mathrm{Sym}(2n)$. Section \ref{maxrank} begins with Theorem \ref{Dnneven} showing that $r_{\mathrm{max}}(W)  < n$ when $n$ is even. The remainder of this section focuses on constructing various C-strings for $W$ and then Lemmas \ref{ip} and \ref{ip2} together with Theorem \ref{rankred} are called upon. Our final section gives a census of abstract regular polytopes for the exceptional finite irreducible Coxeter groups.

\section{Preliminaries} \label{pre}

We recall that a \emph{Coxeter group} is a group $W$ with presentation $\langle s_1, \dots, s_n \mid (s_i s_j)^{m_{ij}} \rangle$ where $m_{ii}=1$ and $m_{ij}=m_{ji} \ge 2$ is a positive integer or infinity for every $1 \leq i < j \leq n$ (infinity being read as no relation).  The \emph{Coxeter rank} of $W$ is $n$.  Associated with this presentation is the \emph{Coxeter diagram} where the nodes correspond to the $s_i$ with a bond between $s_i$ and $s_j$ if $i \not= j$ and $m_{ij} > 2$. If this diagram is connected, then $W$ is said to be an \emph{irreducible} Coxeter group. Coxeter \cite{Coxeter} (see also \cite{coxbible}) classified the finite irreducible Coxeter groups. There are four infinite families denoted $\mathrm{A}_n$, $\mathrm{B}_n$, $\mathrm{D}_n$, $\mathrm{I}_2(m)$, and six exceptional groups $\mathrm{H}_3$, $\mathrm{H}_4$, $\mathrm{F}_4$, $\mathrm{E}_6$, $\mathrm{E}_7$, $\mathrm{E}_8$ (we shall blur the distinction between the group and its root system).

As already mentioned in Section \ref{introduction}, we shall focus on the Coxeter groups $\mathrm{D}_n$ and the exceptional groups. Suppose $W$ is isomorphic to $\mathrm{D}_n$, where $n > 4$. Then $W = SN$ where $S \cong \mathrm{Sym}(n)$ and $N$ is a normal subgroup of $W$ with $N$ an elementary abelian subgroup of order $2^{n-1}$. The subgroup $N$ will sometimes be called the subgroup of (even) sign changes. In this paper we frequently view $W$ as a subgroup of $\mathrm{Sym}(2n)$, as specified in our first lemma.

\begin{lemma}{\cite[(2.10)]{coxbible}} \label{bndn}
Let $\beta_0=(1,n+1)(2,n+2)$ and $\beta_i=(i,i+1)(n+i,n+i+1)$ for every $1 \leq i \leq n-1$  be permutations in $\mathrm{Sym}(2n)$. Then $\langle \beta_0,\beta_1,\dots,\beta_{n-1} \rangle$ is isomorphic to $\mathrm{D}_n$.
\end{lemma}

In a similar fashion, we also have the well-known characterization of $\mathrm{Sym}(n)$.

\begin{lemma}{\cite[(6.4)]{coxbible}} \label{sn}
Let $G$ be a group with presentation $\langle s_1, \dots, s_{n-1} \mid (s_i s_j)^{m_{ij}} \rangle$ where $m_{ii} = 1$, $m_{ij}=2$ if $|i-j| \ge 2$ and $m_{ij}=3$ if $|i-j|=1$ for every $1 \leq i,j \leq n-1$. Then $G$ is isomorphic to $\mathrm{Sym}(n)$.
\end{lemma}

Observe that in Lemma \ref{bndn}, $W \cong \mathrm{D}_n$ is a transitive subgroup of $\mathrm{Sym}(2n)$. Also note, using the notation from Lemma \ref{bndn} that  $\langle \beta_1,\dots,\beta_{n-1} \rangle \cong \mathrm{Sym}(n)$, while products of an even number of transpositions of the form $(i, n+i)$, for some $1 \leq i \leq n$ yields the normal elementary abelian subgroup of sign changes. The next lemma will be used frequently in Section \ref{maxrank}.

\begin{lemma}\label{intN} Suppose that $W \cong \mathrm{D}_n$ with $N$ being the subgroup of sign changes. If $H \leq W$ is such that $W = HN$ and $H \cap N \not\leq Z(W)$, then $W = H$.
\end{lemma}

\begin{proof} Since $N$ is abelian and $W = HN$,  $H \cap N \unlhd W$.  The only non-trivial normal subgroup of $W$ contained in $N$ is $Z(W)$ (if $n$ is even) and $N$. Thus, as $H \cap N \not\leq Z(W)$, $H \cap N = N$ which implies $H = HN = W$.
\end{proof}

Now let $G$ be a group generated by involutions $s_1, \dots, s_r$ with the ordering such that $s_i s_j = s_j s_i$ (or equivalently, $(s_i s_j)^2 = 1$) if $|i-j| \ge 2$ for every $i,j \in I = \{1, \dots, r \}$. Then $(G,\{ s_1, \dots, s_r \})$ is said to be a \emph{string group generated by involutions}.  Set $G_J = \langle s_j \; | \; j \in J \rangle$ for every $J \subseteq I$, and sometimes we may write $G_{i_1,\dots,i_k}$ in place of $G_{\{i_1,\dots,i_k\}}$.  Its \emph{Schl\"{a}fli type} is the sequence $\{p_1,\dots,p_{r-1}\}$ where $p_i$ is the order of $s_is_{i+1}$ for every $1 \leq i \leq r-1$. If a symbol $p_i$ appears $k$ times in adjacent places in the sequence we may write $p_i^k$ in place of writing $p_i$ $k$ times. So a string group generated by involutions will be a string C-group if the intersection condition also holds.

A  set $R$ of elements of a group $G$ is \emph{independent} if for every $g \in R$, $g \not\in \langle R \setminus \{ g \} \rangle$. Let $\mu(G)$ denote the maximum size of an independent subset of $G$.

\begin{theorem}\label{independ} Suppose that $G \cong \mathrm{Sym}(m)$. Then
\begin{enumerate}
\item [(i)] $\mu(G) \leq m-1$.
\item [(ii)] Assume that $m \geq 5$.  If $(G,S)$ is a string group generated by involutions and $S$ is an independent set of size $m-1$, then $S$ is the set of Coxeter generators for $G$.
\end{enumerate}

\begin{proof}(i) This is a theorem of Whiston \cite{whiston}.
(ii) This follows from Cameron and Cara \cite{p7} which classifies independent sets of size $m - 1$ in $\mathrm{Sym}(m)$ when $m \geq 7$, while $m = 5, 6$ can be checked using \textsc{Magma} \cite{magma}.
\end{proof}

\end{theorem}

Theorem \ref{independ} is deployed in the proof of Theorem \ref{Dnneven}, as is our next theorem which gives an upper bound for the rank of a C-string of transitive permutation groups. This is due to Cameron, Fernandes, Leemans and Mixer and is as follows.

\begin{theorem}{\cite[Theorem 1.2]{transrank}} \label{transrank}
Let $G$ be a string C-group of rank $r$ which is isomorphic to a transitive subgroup of $\mathrm{Sym}(n)$ other than $\mathrm{Sym}(n)$ or $\mathrm{Alt}(n)$. Then one of the following holds.
\begin{enumerate}
\item [(i)] $r \leq n/2$.
\item [(ii)]$n \equiv 2 \pmod 4$ and $G \cong \mathrm{B}_n \cong C_2 \wr \mathrm{Sym}(n/2)$ with Schl\"{a}fli type $\{2,3^{n/2-2},4\}$.
\item [(iii)] $G$ is imprimitive and is one of the following.
\begin{enumerate}
\item [(a)]$G=\mathrm{6T9}$ with Schl\"{a}fli type $\{3,2,3\}$.
\item [(b)]$G=\mathrm{6T11}$ with Schl\"{a}fli type $\{2,3,3\}$.
\item [(c)]$G=\mathrm{6T11}$ with Schl\"{a}fli type $\{2,3,4\}$.
\item [(d)]$G=\mathrm{8T45}$ with Schl\"{a}fli type $\{3,4,4,3\}$.
\end{enumerate}
Here, $\mathrm{nTj}$ is the $j$-th conjugacy class among transitive subgroups of $\mathrm{Sym}(n)$ according to Butler and McKay \cite{butlermckay}.
\item [(iv)]$G$ is primitive. In this case, $n=6$ and $G$ is obtained from the degree six permutation representation of $\mathrm{Sym}(5) \cong \mathrm{PGL}_2(5)$ and is the 4-simplex of Schl\"{a}fli type $\{3,3,3\}$.
\end{enumerate}
\end{theorem}

We will also make use of the following lemmas to verify that the generating set of involutions we construct indeed give us string C-groups.

\begin{lemma}{\cite[2E16(a) and 11A10]{arpbible}} \label{ip}
Let $(G,\{ s_1,\dots,s_r \})$ be a string group generated by involutions where $G_{1,\dots,r-1}$ and $G_{2,\dots,r}$ are string C-groups. Then $G$ is a string C-group if either of the following conditions are satisfied.
\begin{enumerate}
\item[(i)] $G_{1,\dots,r-1} \cap G_{2,\dots,r} = G_{2,\dots,r-1}$.
\item[(ii)] $s_r \notin G_{1,\dots,r-1}$ and $G_{2,\dots,r-1}$ is maximal in $G_{2,\dots,r}$. 
\end{enumerate}
\end{lemma}

\begin{lemma}{\cite[2E16(b)]{arpbible}} \label{ip2}
Let $(G,\{ s_1,\dots,s_r \})$ be a string group generated by involutions such that $G_{1,\dots,r-1}$ is a string C-group and $G_{1,\dots,r-1} \cap G_{k,\dots,r} = G_{k,\dots,r-1}$ for every $2 \leq k \leq r-1$. Then $G$ is also a string C-group.
\end{lemma}


Once we have constructed a C-string of maximal rank, we will use the following rank reduction theorem by Brooksbank and Leemans to obtain C-strings of smaller ranks. 

\begin{theorem}{\cite[Theorem 1.1 and Corollary 1.3]{rankred}} \label{rankred}
Let $(G,\{ s_1,\dots,s_n \})$ be a non-degenerate string C-group of rank $n \ge 4$ with Schl\"{a}fli type $\{p_1,\dots,p_{n-1}\}$. If $s_1 \in \langle s_1 s_3 , s_4 \rangle$, then $(G,\{ s_2, s_1 s_3, s_4, \dots,  s_n \})$ is a string C-group of rank $n-1$. In particular, $G$ has a C-string of rank $n-i$ for every $0 \leq i \leq t$, where
\begin{equation*}
t = \max \{ j \in \{0,\dots,n-3\} \mid \text{for all} \;  i \in \{0,\dots,j\}, \ p_{3+i} \text{ is odd} \}.
\end{equation*}
\end{theorem}

\section{C-Strings for $\mathrm{D}_n$} \label{maxrank} 

\begin{theorem}\label{Dnneven}
Let $n$ be even with $n \geq 6$. If $W \cong \mathrm{D}_n$, then $r_{\mathrm{max}}(W) < n$.
\end{theorem}

\begin{proof} By Lemma \ref{bndn} we may regard $W$ as a subgroup of $\mathrm{Sym}(2n)$ and hence, appealing to Theorem \ref{transrank}, $r_{\mathrm{max}}(W) \leq n$. So we must show that $W$ has no C-strings of rank $n$. Assuming that $W$ has a C-string $S$ with $|S| = n$, we seek a contradiction. Let $S = \{s_1, \dots , s_n \}$ with $s_i s_j =s_j s_i$ for all $i, j \in I = \{1, \dots , n \}$ with $|i - j| \geq 2$.  Put $\overline{W} = W/N$ where $N$ is the subgroup of $W$ consisting of the even sign changes, and we use the bar convention. So $\overline{W} \cong \mathrm{Sym}(n)$.\\

(\ref{Dnneven}.1) For $i \in I, s_i \notin Z(W)$.\\

Suppose that $s_i \in Z(W)$ and set $H = \langle S \setminus \{s_i \} \rangle$.  Then $W/Z(W) = HZ(W)/Z(W)$ and therefore $HZ(W) = W = HN$ with $H \cap N \not\leq Z(W)$. Thus $W = H$ by Lemma \ref{intN}, against $S$ satisfying the intersection condition. This proves (\ref{Dnneven}.1).\\

Let $T \subseteq S$ be such that $\overline{T}$ is an independent generating set for $\overline{W}$. Hence $|\overline{T}| \leq n - 1$ by Theorem \ref{independ}(i). Further $\overline{T}$ is a connected string in $\overline{W}$. For if not then $\overline{T} = \overline{T}_1 \cup \overline{T}_2$, $T_i \subseteq S$ with $[\overline{T}_1, \overline{T}_2] = 1$ and $\overline{T}_1 \not= \varnothing \not= \overline{T}_2$. But then $\langle \overline{T}_1 \rangle \langle \overline{T}_2 \rangle =  \overline{W}$ with $[\langle \overline{T}_1 \rangle ,  \langle \overline{T}_2 \rangle ] = 1$, contrary to $\overline{W} \cong \mathrm{Sym}(n)$.\\

(\ref{Dnneven}.2) $|\overline{T}| = n - 1$ or $n - 2$. Moreover, if $|\overline{T}| = n - 2$, then $\overline{T} = \{ \overline{s}_2,  \dots , \overline{s}_{n - 1} \}$.\\

Suppose that $|\overline{T}| \leq n - 3$. Then there exists $s \in S$ such that $[T,s] = 1$. From $\langle \overline{T} \rangle = \overline{W}$, we get $\overline{s} \in Z(\overline{W})$ and then, as $\overline{W} \cong \mathrm{Sym}(n)$, we have $s \in N$. Hence, using (\ref{Dnneven}.1) and Lemma \ref{intN}, $W = \langle T \cup \{s\} \rangle$ which is impossible as $S$ satisfies the intersection property. Thus $|\overline{T}| = n - 1$ or $n - 2$.  A similar argument applies to show $\overline{T} = \{ \overline{s}_2, \dots \overline{s}_{n - 1} \}$ when $|\overline{T}| = n - 2$.\\

(\ref{Dnneven}.3) $|\overline{T}| \not= n - 1$.\\

Suppose that $|\overline{T}| = n - 1$.  We may assume, as $\overline{T}$ is connected, that $\overline{T} = \{ \overline{s}_1,  \dots , \overline{s}_{n - 1} \}$. Put $\overline{X} = \langle \overline{s}_1,  \dots , \overline{s}_{n - 2} \rangle$. By Theorem \ref{independ}(ii) $\overline{T}$ consists of the Coxeter generators for $\overline{W}$. Thus $\overline{X} \cong \mathrm{Sym}(n-1)$. Because $[s_n,s_i] =1$ for $i =1, \dots, n-2$, $\overline{s}_n \in C_{\overline{W}}(\overline{X})$. Since $C_{\overline{W}}(\overline{X}) = 1$, $s_n \in N$. Now $n - 1$ being odd implies $C_N(\overline{X}) = Z(W)$ giving $s_n \in Z(W)$, against (\ref{Dnneven}.1). This rules out $|\overline{T}| = n - 1$.\\

(\ref{Dnneven}.4) $|\overline{T}| \not= n - 2$.\\

Assume that $|\overline{T}| = n - 2$. By (\ref{Dnneven}.2) we have  $\overline{T} = \{ \overline{s}_2,  \dots , \overline{s}_{n - 1} \}$. Put $X = \langle s_2, \dots , s_{n-1} \rangle$ and $X_1 = \langle s_1,s_2, \dots , s_{n-1} \rangle$. Then $W = XN = X_1 N$ and, as $X \leq X_1, X_1 = X(X_1 \cap N)$. If $X_1 \cap N = 1$, then $X_1 \cong \mathrm{Sym}(n)$ and so $X=X_1$. But this contradicts the intersection condition. Thus $X_1 \cap N \not= 1$.  If $X_1 \cap N \not= Z(W)$, then Lemma \ref{intN} forces $X_1 =W$, also a contradiction. Hence $X_1 \cap N = Z(W)$ and consequently $X_1 = XZ(W)$. A similar argument applies for $X_n = \langle s_2, \dots , s_{n-1}, s_n \rangle$ to also yield $X_n = XZ(W)$. But then 
$$W = \langle s_1, \dots,  s_n \rangle \leq XZ(W) \not= W,$$
so proving (\ref{Dnneven}.4).\\

Combining (\ref{Dnneven}.2), (\ref{Dnneven}.3) and (\ref{Dnneven}.4) yields the desired contradiction, so establishing Theorem \ref{Dnneven}.
\end{proof}

In the remainder of this section we construct examples of C-strings for $\mathrm{D}_n$.

For every integer $n \ge 5$, we define $t_1,\dots,t_n$ in $\mathrm{Sym}(2n)$ as follows.
\begin{align*}
\begin{split}
t_1&=\prod_{j=2}^{n}(j,n+j), \\
t_i&=(i-1,i)(n+i-1,n+i) \text{ for every } 2 \leq i \leq n.
\end{split}
\end{align*}

\begin{lemma} \label{csp1}
For every $1 \leq i,j \leq n$, we have
\begin{equation*}
\text{order of} \; t_i t_j = \begin{cases}
1, & \text{ if } \ i=j, \\
2, & \text{ if } \ |i-j| \ge 2, \\
3, & \text{ if } \ |i-j| = 1 \text{ and } \{i,j\} \neq \{1,2\}, \\
4, & \text{ if } \ \{i,j\} = \{1,2\}.
\end{cases}.
\end{equation*}
\end{lemma}

\begin{proof}
Each of the elements $t_1,\dots,t_n$ are involutions as they are defined as products of pairwise disjoint transpositions. Note that
\begin{align*}
t_1 t_2 &=(1,2,n+1,n+2) \prod_{j=3}^n (j,n+j), \\
t_i t_{i+1}&=(i-1,i+1,i)(n+i-1,n+i+1,n+i) \text{ for every } 2 \leq i \leq n.
\end{align*}
For every $2 \leq i,j \leq n$ with $|i-j|\ge2$, $t_i$ and $t_j$ are disjoint, whereas for every $3 \leq k \leq n$, note that 
\begin{equation*}
t_1 t_k = (k-1,n+k)(k,n+k-1) \prod_{\substack{2 \leq j \leq n \\ j \notin \{k-1,k\}}} (j,n+j).
\end{equation*}
As such, $t_i t_j$ is a product of pairwise disjoint transpositions for every $1 \leq i,j \leq n$ with $|i-j|\ge2$. The lemma follows from examining each of the products $t_i t_j$ in disjoint cycle notation.
\end{proof}

\begin{lemma} \label{csp2} 
$\langle t_1,\dots,t_n \rangle$ is a string C-group.
\end{lemma}

\begin{proof}
We prove by induction on $k$ that $\langle t_1,\dots,t_k\rangle$ is a string C-group for every $3 \leq k \leq n$. For the base case $k=3$, it is easy to verify that $\langle t_1,t_2 \rangle \cap \langle t_2,t_3 \rangle = \langle t_2 \rangle$. So now suppose that $k>3$ and that $\langle t_1,\dots,t_{k-1}\rangle$ is a string C-group. Now by Lemmas \ref{sn} and \ref{csp1}, $\langle t_2,\dots,t_k \rangle \cong \mathrm{Sym}(k)$ is a string C-group and $\langle t_2,\dots,t_{k-1}\rangle \cong \mathrm{Sym}(k-1)$ which is maximal in $\langle t_2,\dots,t_k \rangle$. We also have $t_k \notin \langle t_1,\dots,t_{k-1}\rangle$ because $t_k$ moves $k-1$ to $k$, which are in different $\langle t_1,\dots,t_k\rangle$-orbits. It follows from Lemma \ref{ip} that $\langle t_1,\dots,t_k\rangle$ is a string C-group.
\end{proof}

\begin{lemma} \label{csp3}
Let $n$ be odd and $n \ge 5$. Then $\langle t_1,\dots,t_n \rangle$ is isomorphic to $\mathrm{D}_n$.
\end{lemma}

\begin{proof}
Let $G=\langle t_1,\dots,t_n \rangle$, $H=\langle \beta_0,\beta_1,\dots,\beta_{n-1} \rangle$ where $\beta_0=(1,n+1)(2,n+2)$ and $\beta_i=(i,i+1)(n+i,n+i+1)$ for every $1 \leq i \leq n-1$. Then $H \cong \mathrm{D}_n$ by Lemma \ref{bndn} and we want to show that $G=H$. Since $t_i=\beta_{i-1}$ for every $2 \leq i \leq n$, it suffices to show that $\beta_0 \in G$ and $t_1 \in H$.
\par We note that $$t_1 t_2 = (1,2,n+1,n+2)\prod_{j=3}^n (j,n+j),$$ whence $\beta_0 = (1,n+1)(2,n+2)=(t_1 t_2)^2 \in G$. Now let $\alpha_{i,j}=(i,n+i)(j,n+j)$ and $\gamma_{i,j}=(i,j)(n+i,n+j)$ for every $2 \leq i,j \leq n$.  Note that $\alpha_{2,3}=\beta_0^{\beta_2\beta_1} \in H$. Now for $4 \leq k \leq n$, we have $$\gamma_{2,k}={\beta_2}^{\prod_{j=4}^k \beta_{j-1}} \text{ and } \gamma_{3,k}={\beta_2}^{\gamma_{2,k}}.$$ It follows that for every $4 \leq i,j \leq n$, we have $\alpha_{i,j}={\alpha_{2,3}}^{\gamma_{2,i}\gamma_{3,j}} \in H$ and hence, $t_1=\alpha_{2,3}\alpha_{4,5}\cdots\alpha_{n-1,n} \in H$.
\end{proof}

\begin{theorem}\label{nodd}
Let $n$ be odd and $n \ge 5$. Then $\{t_1,\dots,t_n\}$ is a rank $n$ C-string of $\mathrm{D}_n$ with Schl\"{a}fli type $\{4,3^{n-2}\}$. Therefore, $r_{\mathrm{max}}(\mathrm{D}_n)=n$ and there is a rank $r$ C-string of $\mathrm{D}_n$ for every $3 \leq r \leq n$.
\end{theorem}

\begin{proof}
The first assertion follows from Lemmas \ref{csp1}, \ref{csp2} and \ref{csp3}. The second and third assertion follows from Theorem \ref{transrank} and  Theorem \ref{rankred}, respectively.
\end{proof}

To describe our C-strings, we will use an interesting graph called the string C-group permutation representation (CPR) graph, first introduced in \cite{cpr}. Let $(G,\{ s_1, \dots, s_r \})$ be a string C-group identified as a permutation group of degree $m$. Then the \emph{CPR graph} of $G$ is an $r$-edge-labelled multigraph with vertex set $\{1,\dots,m\}$ such that $\{i,j\}$ is an edge of label $k$ if $s_k$ moves $i$ to $j$. Recall the notation for $W \cong \mathrm{D}_n$ when $W \leq \mathrm{Sym}(2n)$ as described in Section \ref{pre}. That is, $W = SN$ where $S \cong \mathrm{Sym}(n)$ and $N$ is the subgroup of even sign changes.

\begin{theorem}{\cite[Lemma 21]{p20}} \label{skeleton}
Let $n \ge 5$ and $3 \leq d \leq n-2$. Then $\mathrm{Sym}(n)$ has a C-string of rank $d$ with Schl\"{a}ffli type $\{3^{d-3},6,n-d+2\}$ whose CPR graph can be obtained by appending $n-d$ vertices and edges with labels $d-1$ and $d$ alternatively to the CPR graph of the $d$-simplex, as follows. 
\begin{center}
\begin{tikzpicture}[xscale=1.4]
	\vx (a) at (-4,0) {}; \node at (-3.5,0.25) {$1$};
	\vx (b) at (-3,0) {}; \node at (-2.5,0.25) {$2$};
	\vx (c) at (-2,0) {}; \node at (-1.5,0) {$\boldsymbol{\cdots}$};
	\vx (d) at (-1,0) {}; \node at (-0.5,0.25) {$d-1$};
	\vx (e) at (0,0) {}; \node at (0.5,0.25) {$d$};
	\vx (A) at (1,0) {}; \node at (1.5,0.25) {$d-1$};
	\vx (B) at (2,0) {}; \node at (2.5,0.25) {$d$};
	\vx (x) at (3,0) {}; \node at (3.5,0) {$\boldsymbol{\cdots}$};
	\vx (y) at (4,0) {};
  \draw (a) -- (b) -- (c) (d) -- (e) -- (A) -- (B) -- (x);
\end{tikzpicture}
\end{center}
\end{theorem}

\begin{lemma} \label{ipdn}
Let $G=\langle s_1,\dots, s_r\rangle \leq W \leq \mathrm{Sym}(2n)$ with $W \cong \mathrm{D}_n$ be a string group generated by involutions. If $G_{1,\dots,r-1}=S$ is a string C-group and $s_r \in N \setminus Z(G)$, then $G=W$ and is also a string C-group.
\end{lemma}

\begin{proof}
By Lemma \ref{intN} we have $G=W$. Let $2 \leq k \leq r-1$. Since $s_r \in N$, $G_{k,\dots,r} \leq G_{k,\dots,r-1}N$, and so $G_{k,\dots,r}=G_{k,\dots,r-1}(G_{k,\dots,r} \cap N)$ by the Dedekind law. Then by Dedekind law again,
\begin{eqnarray*}
    G_{1,\dots,r-1} \cap G_{k,\dots,r} &=& G_{1,\dots,r-1} \cap (G_{k,\dots,r-1}(G_{k,\dots,r} \cap N)) \\
    &=& G_{k,\dots,r-1}(G_{1,\dots,r-1} \cap G_{k,\dots,r} \cap N) \\
    &=& G_{k,\dots,r-1},
\end{eqnarray*}
as $G_{1,\dots,r-1} \cap N = S \cap N = 1$. Thus the lemma follows from Lemma \ref{ip2}.
\end{proof}

Let $n \ge 6$ be an even integer for the remainder of this section. We first construct rank three C-strings of $\mathrm{D}_n$, defining $t_1,t_2,t_3$ in $\mathrm{Sym}(2n)$ as follows.
\begin{align*}
\begin{split}
t_1 &= (1,2)(n+1,n+2)(n-1,2n-1)(n,2n), \\
t_2 &= (2,3)(4,5)\cdots(n-2,n-1)(n+2,n+3)(n+4,n+5)\cdots(2n-2,2n-1), \\
t_3 &= (3,4)(5,6)\cdots(n-1,n)(n+3,n+4)(n+5,n+6)\cdots(2n-1,2n).
\end{split}
\end{align*}

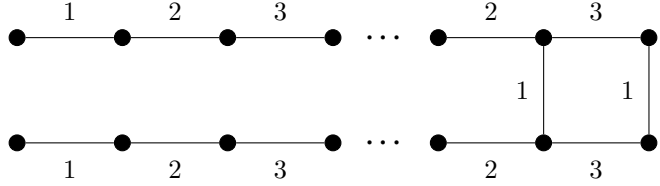
\begin{center}
\begin{tikzpicture}[xscale=1.4,yscale=1.4]
	\vx (a) at (-3,0) {}; \node at (-2.5,0.25) {$1$};
	\vx (b) at (-2,0) {}; \node at (-1.5,0.25) {$2$};
	\vx (c) at (-1,0) {}; \node at (-0.5,0.25) {$3$};
	\vx (x) at (0,0) {}; \node at (0.5,0) {$\boldsymbol{\cdots}$};
	\vx (A) at (1,0) {}; \node at (1.5,0.25) {$2$};
	\vx (B) at (2,0) {}; \node at (2.5,0.25) {$3$};
	\vx (C) at (3,0) {};
  \draw (a) -- (b) -- (c) -- (x) (A) -- (B) -- (C);
	\vx (-a) at (-3,-1) {}; \node at (-2.5,-1.25) {$1$};
	\vx (-b) at (-2,-1) {}; \node at (-1.5,-1.25) {$2$};
	\vx (-c) at (-1,-1) {}; \node at (-0.5,-1.25) {$3$};
	\vx (-x) at (0,-1) {}; \node at (0.5,-1) {$\boldsymbol{\cdots}$};
	\vx (-A) at (1,-1) {}; \node at (1.5,-1.25) {$2$};
	\vx (-B) at (2,-1) {}; \node at (2.5,-1.25) {$3$};
	\vx (-C) at (3,-1) {};
  \draw (-a) -- (-b) -- (-c) -- (-x) (-A) -- (-B) -- (-C);
  	\node at (1.8,-0.5) {$1$}; \node at (2.8,-0.5) {$1$};
  \draw (B) -- (-B) (C) -- (-C);
\end{tikzpicture}
\end{center}

\begin{lemma} \label{csr1}
$\{t_1,t_2,t_3\}$ is a rank 3 C-string of $\mathrm{D}_n$ with Schl\"{a}fli type $\{12,n-1\}$.
\end{lemma}

\begin{proof}
Each of the elements $t_1,t_2,t_3$ are involutions as they are defined as products of pairwise disjoint transpositions. Note that
\begin{align*}
t_1 t_2 &= (1,3,2)(n-1,2n-2,2n-1,n-2)(n,2n) \\ &\phantom{=}\quad (4,5)\cdots(n-4,n-3)(n+4,n+5)\cdots(2n-4,2n-3), \\
t_2 t_3 &= (2,4,\dots,n,n-1,n-3,\dots,3)(n+2,n+4,\dots,2n,2n-1,2n-3,\dots,n+3), \\
t_1 t_3 &= (1,2)\cdots(n-3,n-2)(n-1,2n)(n,2n-1).
\end{align*}
Since $\langle t_1,t_2 \rangle \cong \mathrm{Dih}(24)$, its elements are of the form ${t_1}^\varepsilon (t_1 t_2)^k$ for some integers $\varepsilon \in \{0,1\}$ and $0 \leq k \leq 11$. Also since $\langle t_2,t_3 \rangle \leq S \cap \mathrm{Stab}_{\mathrm{Sym}(2n)}(1)$, we have
\begin{eqnarray*}
\langle t_1,t_2 \rangle \cap \langle t_2,t_3 \rangle &\leq& \langle t_1,t_2 \rangle \cap \big(S \cap \mathrm{Stab}_{\mathrm{Sym}(2n)}(1)\big) \\
&=& (\langle t_1,t_2 \rangle \cap S) \cap \big(\langle t_1,t_2 \rangle \cap \mathrm{Stab}_{\mathrm{Sym}(2n)}(1) \big) \\
&=& \langle t_1,(t_1 t_2)^4 \rangle \cap \langle t_1,(t_1 t_2)^3 \rangle = \langle t_2 \rangle.
\end{eqnarray*}
It follows that $\langle t_1,t_2 \rangle \cap \langle t_2,t_3 \rangle = \langle t_2 \rangle$ and so $\langle t_1,t_2,t_3 \rangle$ is a string C-group with Schl\"{a}ffli type $\{12,n-1\}$. We now show that $H=\langle t_1,t_2,t_3 \rangle$ is isomorphic to $\mathrm{D}_n$. Note that $\overline{H}=S$ by Theorem \ref{skeleton} and so $HN=\mathrm{D}_n$. Since $$(n-2,2n-2)(n-1,2n-1)=(t_1 t_2)^6 \in N$$ and  $(n-2,2n-2)(n-1,2n-1) \notin Z(\mathrm{D}_n)$. Then calling upon Lemma \ref{intN} yields $H=\mathrm{D}_n$.
\end{proof}

We now construct rank $r$ C-strings of $\mathrm{D}_n$ for every $4 \leq r \leq n-1$ as follows. We use Theorem \ref{skeleton} to construct a rank $r-1$ C-string for the subgroup $S \cong \mathrm{Sym}(n)$ of $\mathrm{D}_n$, and then append an element in $N$. We note that the construction depends on the parity of $r$.

If $r$ is odd, we define $t_1,\dots,t_r$ in $\mathrm{Sym}(2n)$ as follows.

\begin{align*}
\begin{split}
t_i &= (i,i+1)(n+i,n+i+1) \ \text{for every $1 \leq i \leq r-3$}, \\
t_{r-2} &= (r-2,r-1)(r,r+1)\cdots(n-1,n) \\ &\phantom{=}\quad (n+r-2,n+r-1)(n+r,n+r+1)\cdots(2n-1,2n), \\
t_{r-1} &= (r-1,r)(r+1,r+2)\cdots(n-2,n-1) \\ &\phantom{=}\quad (n+r-1,n+r)(n+r+1,n+r+2)\cdots(2n-2,2n-1), \\
t_r &= (n-1,2n-1)(n,2n).
\end{split}
\end{align*}

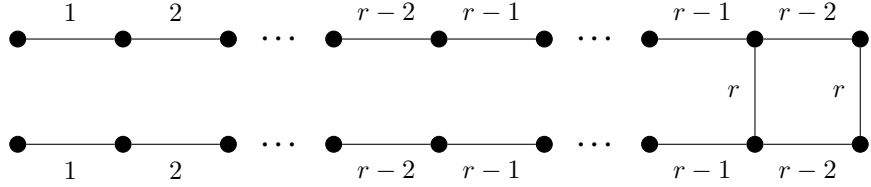
\begin{center}
\begin{tikzpicture}[xscale=1.4,yscale=1.4]
	\vx (a) at (-4,0) {}; \node at (-3.5,0.25) {$1$};
	\vx (b) at (-3,0) {}; \node at (-2.5,0.25) {$2$};
	\vx (c) at (-2,0) {}; \node at (-1.5,0) {$\boldsymbol{\cdots}$};
	\vx (d) at (-1,0) {}; \node at (-0.5,0.25) {$r-2$};
	\vx (e) at (0,0) {}; \node at (0.5,0.25) {$r-1$};
	\vx (x) at (1,0) {}; \node at (1.5,0) {$\boldsymbol{\cdots}$};
	\vx (A) at (2,0) {}; \node at (2.5,0.25) {$r-1$};
	\vx (B) at (3,0) {}; \node at (3.5,0.25) {$r-2$};
	\vx (C) at (4,0) {};
  \draw (a) -- (b) -- (c) (d) -- (e) -- (x) (A) -- (B) -- (C);
	\vx (-a) at (-4,-1) {}; \node at (-3.5,-1.25) {$1$};
	\vx (-b) at (-3,-1) {}; \node at (-2.5,-1.25) {$2$};
	\vx (-c) at (-2,-1) {}; \node at (-1.5,-1) {$\boldsymbol{\cdots}$};
	\vx (-d) at (-1,-1) {}; \node at (-0.5,-1.25) {$r-2$};
	\vx (-e) at (0,-1) {}; \node at (0.5,-1.25) {$r-1$};
	\vx (-x) at (1,-1) {}; \node at (1.5,-1) {$\boldsymbol{\cdots}$};
	\vx (-A) at (2,-1) {}; \node at (2.5,-1.25) {$r-1$};
	\vx (-B) at (3,-1) {}; \node at (3.5,-1.25) {$r-2$};
	\vx (-C) at (4,-1) {};
  \draw (-a) -- (-b) -- (-c) (-d) -- (-e) -- (-x) (-A) -- (-B) -- (-C);
  	\node at (2.8,-0.5) {$r$}; \node at (3.8,-0.5) {$r$};
  \draw (B) -- (-B) (C) -- (-C);
\end{tikzpicture}
\end{center}

If $r$ is even, we define $t_1,\dots,t_r$ in $\mathrm{Sym}(2n)$ as follows.

\begin{align*}
\begin{split}
t_i &= (i,i+1)(n+i,n+i+1) \ \text{for every $1 \leq i \leq r-3$}, \\
t_{r-2} &= (r-2,r-1)(r,r+1)\cdots(n-2,n-1)(n+r-2,n+r-1) \\ &\phantom{=}\quad (n+r,n+r+1)\cdots(2n-2,2n-1), \\
t_{r-1} &= (r-1,r)(r+1,r+2)\cdots(n-1,n)(n+r-1,n+r)\cdots(2n-1,2n), \\
t_r &= (n-2,2n-2)(n-1,2n-1).
\end{split}
\end{align*}

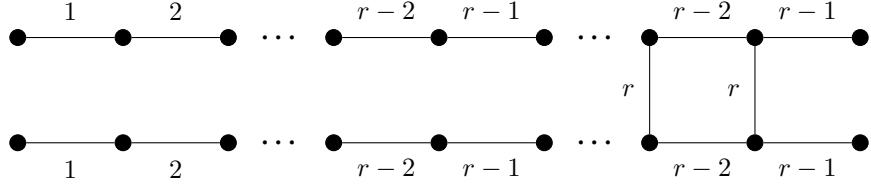
\begin{center}
\begin{tikzpicture}[xscale=1.4,yscale=1.4]
	\vx (a) at (-4,0) {}; \node at (-3.5,0.25) {$1$};
	\vx (b) at (-3,0) {}; \node at (-2.5,0.25) {$2$};
	\vx (c) at (-2,0) {}; \node at (-1.5,0) {$\boldsymbol{\cdots}$};
	\vx (d) at (-1,0) {}; \node at (-0.5,0.25) {$r-2$};
	\vx (e) at (0,0) {}; \node at (0.5,0.25) {$r-1$};
	\vx (x) at (1,0) {}; \node at (1.5,0) {$\boldsymbol{\cdots}$};
	\vx (A) at (2,0) {}; \node at (2.5,0.25) {$r-2$};
	\vx (B) at (3,0) {}; \node at (3.5,0.25) {$r-1$};
	\vx (C) at (4,0) {};
  \draw (a) -- (b) -- (c) (d) -- (e) -- (x) (A) -- (B) -- (C);
	\vx (-a) at (-4,-1) {}; \node at (-3.5,-1.25) {$1$};
	\vx (-b) at (-3,-1) {}; \node at (-2.5,-1.25) {$2$};
	\vx (-c) at (-2,-1) {}; \node at (-1.5,-1) {$\boldsymbol{\cdots}$};
	\vx (-d) at (-1,-1) {}; \node at (-0.5,-1.25) {$r-2$};
	\vx (-e) at (0,-1) {}; \node at (0.5,-1.25) {$r-1$};
	\vx (-x) at (1,-1) {}; \node at (1.5,-1) {$\boldsymbol{\cdots}$};
	\vx (-A) at (2,-1) {}; \node at (2.5,-1.25) {$r-2$};
	\vx (-B) at (3,-1) {}; \node at (3.5,-1.25) {$r-1$};
	\vx (-C) at (4,-1) {};
  \draw (-a) -- (-b) -- (-c) (-d) -- (-e) -- (-x) (-A) -- (-B) -- (-C);
  	\node at (1.8,-0.5) {$r$}; \node at (2.8,-0.5) {$r$};
  \draw (A) -- (-A) (B) -- (-B);
\end{tikzpicture}
\end{center}

\begin{lemma} \label{csr2}
Let $n$ be even and $n \ge 6$. Then for every $4 \leq r \leq n-1$, $\{t_1,\dots,t_r\}$ is a rank $r$ C-string of $\mathrm{D}_n$ with Schl\"{a}fli type $\{3^{r-4},6,n-r+3,4\}$.
\end{lemma}

\begin{proof}
Each of the elements $t_1,\dots,t_r$ are involutions as they are defined as products of pairwise disjoint transpositions. Note that for every $1 \leq k \leq r-3$, we have
\begin{eqnarray*}
t_k t_r = 
\begin{cases}
(k,k+1)(n+k,n+k+1)(n-1,2n-1)(n,2n) \ &, \text{ if $r$ is odd } \\ 
(k,k+1)(n+k,n+k+1)(n-2,2n-2)(n-1,2n-1) \ &, \text{ if $r$ is even } \\
\end{cases},
\end{eqnarray*}
which are products of pairwise disjoint transpositions because 
\begin{eqnarray*}
    k+1 \leq (r-3)+1 \leq ((n-1)-3)+1 = n-3 < n-2.
\end{eqnarray*}
Likewise, if $r$ is odd we have
\begin{eqnarray*}
t_{r-2} t_r &=& (n-1,2n)(n,2n-1) (r-2,r-1)(r,r+1)\cdots(n-3,n-2) \\ &\phantom{=}& \quad(n+r-2,n+r-1)(n+r,n+r+1)\cdots(2n-3,2n-2),
\end{eqnarray*}
whereas if $r$ is even we have
\begin{eqnarray*}
t_{r-2} t_r &=& (n-2,2n-1)(n-1,2n-2)(r-2,r-1)(r,r+1)\cdots(n-4,n-3) \\ &\phantom{=}&\quad (n+r-2,n+r-1)(n+r,n+r+1)\cdots(2n-4,2n-3),
\end{eqnarray*}
which are also products of pairwise disjoint transpositions. Finally, we have
\begin{eqnarray*}
t_{r-1} t_r &=& (n-1,n-2,2n-1,2n-2)(n,2n) \\ &\phantom{=}& \quad(r-1,r)(r+1,r+2)\cdots(n-4,n-3) \\ &\phantom{=}& \quad (n+r-1,n+r)(n+r+1,n+r+2)\cdots(2n-4,2n-3)
\end{eqnarray*}
for $r$ odd and we also have
\begin{eqnarray*}
t_{r-1} t_r &=& (n-2,n-3,2n-2,2n-3)(n-1,n,2n-1,2n) \\ &\phantom{=}&\quad (r-1,r)(r+1,r+2)\cdots(n-5,n-4) \\ &\phantom{=}&\quad (n+r-1,n+r)(n+r+1,n+r+2)\cdots(2n-5,2n-4),
\end{eqnarray*}
for $r$ even. It follows from Lemma \ref{ipdn} that $\langle t_1,\dots,t_r \rangle=\mathrm{D}_n$ is a string C-group with Schl\"{a}fli type $\{3^{r-4},6,n-r+3,4\}$.
\end{proof}

\begin{theorem}\label{neven}
Let $n$ be even and $n \ge 6$. Then $r_{\mathrm{max}}(\mathrm{D}_n)=n-1$ and there is a rank $r$ C-string of $\mathrm{D}_n$ for every $3 \leq r \leq n-1$.
\end{theorem}

\begin{proof}
The theorem follows from Lemmas \ref{csr1}, \ref{csr2} and Theorem \ref{Dnneven}.
\end{proof}

Together Theorems \ref{Dnneven}, \ref{nodd} and \ref{neven} prove Theorems \ref{maximalrank} and \ref{intermediateranks}.

\section{Exceptional Finite Irreducible Coxeter Groups} \label{census}

The following table lists the number of abstract regular polytopes (up to isomorphism and duality with the number of self-dual polytopes in brackets) for each exceptional Coxeter groups computed using \textsc{Magma} \cite{magma}.

\begin{table}[h!] \caption{Number of polytopes for each exceptional Coxeter group.} 
\centering
\begin{tabular}{|c|c|c|c|c|c|c|c|} \hline
Group & Total & Rank 3 & Rank 4 & Rank 5 & Rank 6 & Rank 7 & Rank $\ge8$ \\ \hline
$\mathrm{H}_3$ & 8(1) & 8(1) & 0 & 0 & 0 & 0 & 0 \\ \hline
$\mathrm{H}_4$ & 59(6) & 45(2) & 14(4) & 0 & 0 & 0 & 0 \\ \hline
$\mathrm{F}_4$ & 5(1) & 3(0) & 2(1) & 0 & 0 & 0 & 0 \\ \hline
$\mathrm{E}_6$ & 147(18) & 87(12) & 50(4) & 10(2) & 0 & 0 & 0 \\ \hline
$\mathrm{E}_7$ & 3662(10) & 1577(10) & 1525(0) & 465(0) & 95(0) & 0 & 0 \\ \hline
$\mathrm{E}_8$ & 11689(142) & 6746(117) & 3584(22) & 986(2) & 310(0) & 63(1) & 0 \\ \hline
\end{tabular}
\end{table}


\end{document}